\documentclass[11pt]{amsart}

\usepackage[english]{babel}     
\usepackage[date=year,backend=bibtex,doi=false,isbn=false,url=false,eprint=false, maxnames=50]{biblatex}
\usepackage[utf8]{inputenc}
\usepackage{amsmath}
\usepackage{amsfonts}
\usepackage{amssymb}

\let\emptyset\varnothing
\usepackage{mathrsfs}
\usepackage{stmaryrd}
\usepackage{hyperref}
\usepackage{xcolor}
\usepackage{todonotes}
\usepackage{comment}
\usepackage{enumerate}
\usepackage{tikz-cd}
\usepackage[capitalize]{cleveref}
\usepackage{tikz}
\usepackage[T1]{fontenc}

\usetikzlibrary{shapes,calc,arrows,intersections,decorations.pathreplacing,decorations.markings,shapes.geometric,calc,positioning,backgrounds}

\newtheorem{theorem}{Theorem}[section]

\newtheorem{lemma}[theorem]{Lemma}
\newtheorem{proposition}[theorem]{Proposition}
\newtheorem{corollary}[theorem]{Corollary}

\theoremstyle{definition}
\newtheorem{definition}[theorem]{Definition}
\newtheorem{example}[theorem]{Example}
\newtheorem{remark}[theorem]{Remark}

\theoremstyle{remark}
\numberwithin{equation}{section}

\DeclareMathOperator{\qNor}{qNor}

\newcommand{\calZ}{\mathcal{Z}}

\let\epsilon\varepsilon

\DeclareMathOperator{\Dist}{Dist}
\DeclareMathOperator{\Radius}{Radius}

\newcommand{\Link}{\textrm{Link}}
\newcommand{\Star}{\textrm{Star}}

\newcommand{\Nor}{\textrm{Nor}}
\newcommand{\Int}{\mathrm{Int}}
\newcommand{\barotimes}{\overline{\otimes}}

\title[Internal graphs of graph products of hyperfinite II$_1$-factors]{Internal graphs of graph products of hyperfinite II$_1$-factors}

\date{\noindent \today.  MSC2010 keywords:  46L51, 46L54.  MC is  supported by the NWO Vidi grant VI.Vidi.192.018 `Non-commutative harmonic analysis and rigidity of operator algebras'.}

\author[Martijn Caspers]{Martijn Caspers}
\author[Enli Chen]{Enli Chen}

\address{TU Delft, EWI/DIAM,
	P.O.Box 5031,
	2600 GA Delft,
	The Netherlands}

\email{M.P.T.Caspers@tudelft.nl}

\email{E.Chen-1@tudelft.nl}

\begin{document}

\maketitle

\begin{abstract}
In this paper, we show that for a graph $\Gamma$ from a class named H-rigid graphs, its subgraph $\Int(\Gamma)$, named the internal graph of $\Gamma$, is an isomorphism invariant of the graph product of hyperfinite II$_1$-factors $R_{\Gamma}$. In particular, we can  classify $R_{\Gamma}$ for some typical types of graphs,   including lines, cyclic graphs and infinite regular trees. As an application, we also show that for two isomorphic graph products of hyperfinite II$_1$-factors over H-rigid graphs, the difference of the radius between the two graphs will not be larger than 1. Our proof is based on the recent resolution of the Peterson-Thom conjecture. 
\end{abstract}

\section{Introduction}
Graph products of von Neumann algebras were introduced by M\l{}otkowski \cite{mlotkowskiLfreeProbability2004}, as well as Fima together with the first author \cite{caspersGraphProductsOperator2017}. It associates to a graph with a $\sigma$-finite von Neumann algebra labeled by every edge, a new von Neumann algebra that contains all the vertex von Neumann algebras and these vertex von Neumann algebras commute (resp. are freely independent) if and only if the vertices share an edge (resp. do not share an edge). The construction generalizes tensor products, in case of complete graphs, and free products, in case of graphs without edges. 
Graph products have been studied in the context of Popa's deformation/rigidity theory in \cite{ BCC,caspersGraphProductsOperator2017, charlesworthStructureGraphProduct2024,ChifanDaviesDrimbeI, ChifanDaviesDrimbeII}.  This leads to rigidity theorems for graph products with specific structure of the graphs and von Neumann algebras. 

In particular, in \cite{BCC} we established a rigidity theorem for graph products of a class of non-amenable  II$_1$-factors  (Theorem A in \cite{BCC}). In order to do this we introduced the notion of a `rigid graph'.   We then show that for graph products over rigid graphs of this class of non-amenable II$_1$-factors, the rigid graphs are  isomorphism invariants. In \cite{BCC}, one of the key steps of the proof for this rigidity result is to use an embedding theorem for graph product von Neumann algebras (Theorem I in \cite{BCC}). This embedding theorem  only can be applied when the II$_1$-factors for all vertices are  non-amenable. Therefore, in \cite{BCC}, the non-amenablity of the II$_1$-factors is crucial for the proof of the rigidity theorem. 

 It is quite natural to ask if such rigidity theorem also holds for graph products of the amenable II$_1$-factor $R$. That is to say, for two graphs $\Gamma$ and $\Lambda$, if for the graph products we have $R_{\Gamma}\simeq R_{\Lambda}$, then can we show that $\Gamma\simeq \Lambda$, or at least some parts of the two graphs are isomorphic? First we consider two extreme cases, i.e. when the graphs are complete or without edges. When the graphs are complete, the graph products become tensor products. Since $R\barotimes R\simeq R$, we cannot distinguish any two  complete graphs by their graph product of hyperfinite II$_1$-factors. When the graphs have no edges, the graph products become  free products. Then this problem is equivalent  to the free factor problem and is therefore very hard and outside the scope of our paper. On the other hand, we note that it is simply not true that one can distinguish the graph products of hyperfinite II$_1$-factors over two non-isomorphic graphs beyond tensor products. For instance, considering two non-isomorphic bipartite graphs $K_{3,3}$ and $K_{2,5}$, by the R{\u a}dulescu amplifcation formula (formula (0.2) in \cite{Radulescu94}), we have 
 \[
 R_{K_{3,3}}\simeq\mathcal{L}(\mathbb{F}_3)\barotimes\mathcal{L}(\mathbb{F}_3)\simeq \mathcal{L}(\mathbb{F}_3)^{\sqrt{2}}\barotimes\mathcal{L}(\mathbb{F}_3)^{1/\sqrt{2}}\simeq\mathcal{L}(\mathbb{F}_2)\barotimes\mathcal{L}(\mathbb{F}_5)\simeq R_{K_{2,5}},
 \]
 where $\mathcal{L}(\mathbb{F}_n)$ is the group von Neumann algebra of the free group with $n\in \mathbb{N}_{\geq 2}$ generators and $M^t$ denotes the amplification of a II$_1$-factor $M$ with exponent $t \in \mathbb{R}_{>0}$.

 In this paper, we will introduce a class of non-trival graphs, named H-rigid graphs. Then we show that for the graph product of hyperfinite II$_1$-factors $R_{\Gamma}$ over the graph $\Gamma$ from this class, the subgraph $\Int(\Gamma)$ of $\Gamma$  is an isomorphism invariant   (\cref{thm4.7.}). Here $\Int(\Gamma)$ is the subgraph of interior vertices of $\Gamma$, i.e. vertices   whose neighbors do not form a complete graph.     The novelty of our proof for this rigidity result is that it uses the recent celebrated resolution of the Peterson-Thom conjecture \cite{PetersonThom}. This conjecture was solved through Hayes' reduction to a random matrix problem \cite{hayesRandomMatrixApproach2022} that was eventually solved by Belinschi-Capitaine \cite{BelinschiCapitaine22} and Bordenave-Collins \cite{BordenaveCollins23}. Remarkable applications have already been found in \cite{HayesJekelElayavalli}. As a further application, we obtain that for the graph products of hyperfinite II$_1$-factors over some specific types of graphs, like lines $l_n$, cyclic graphs $\mathbb{Z}_n$ or   infinite regular trees, the graphs are isomorphism invariants (\cref{coro4.10}). In other words, we can  classify the II$_1$-factors $R_{l_n}$, $R_{\mathbb{Z}_n}$  and $R_{\tilde{\mathcal{T}}}$ ($\tilde{\mathcal{T}}$ are infinite regular trees). This result also partially answers Conjecture 5.10.5 in \cite{MatthijsBorstThesis}. 
 
 In \cite{BCC} we established a graph radius rigidity result for graph products of group von Neumann algebras of icc groups (Theorem F in \cite{BCC}). We showed that if two graph products of group von Neumann algebras are isomorphic, then the difference between the radius of two graphs will not be larger than 2. In this paper, we will use our new obtained rigidity result to strengthen this graph radius rigidity result for   graph products of hyperfinite II$_1$-factors. We show that the radius difference between two isomorphic graph products of II$_1$-factors over H-rigid graphs must not be larger than 1 (\cref{coro4.11}).

\section{Preliminaries}

For standard theory on von Neumann algebras, we refer to the books \cite{stratilaLecturesNeumannAlgebras2019,Takesaki1}. For  von Neumann algebras $M$ and $N$ we use $M_\ast$ for the predual, $M'$ for the commutant, $M \otimes_{alg} N$ for the algebraic tensor product and $M \overline{\otimes} N = (M \otimes_{alg} N )''$ for the von Neumann algebraic tensor product.   We say that a von Neumann algebra $M$ is finite if it admits a faithful normal tracial state $\tau$. We also refer to the pair $(M, \tau)$ as a tracial von Neumann algebra.   We call $M$ diffuse if $M$ does not contain non-zero minimal projections. A von Neumann subalgebra is always assumed to contain the unit of the larger algebra.

The following is well-known but we have not found its statement in the literature. 

\begin{lemma}\label{Lem=Slicing}
Suppose $M$ and $N$ are two von Neumann algebras, and $A\subseteq M $ and $B\subseteq N$ are von Neumann subalgebras. Let $x\in M\barotimes N$. If for any $\omega\in N_*$, $(id\otimes \omega)(x)\in A$, and for any $\omega'\in M_*$, $(\omega'\otimes id)(x)\in B$, then $x\in A\barotimes B$. 
\end{lemma}
\begin{proof}
For $a\in A'$ and $\omega \in N_\ast$,  we have 
$$(id\otimes \omega)(x(a\otimes 1))=(id\otimes\omega)(x)a=a (id\otimes\omega)(x)=(id\otimes \omega)((a\otimes 1)x).$$
Then we have 
$x(a\otimes 1)=(a\otimes 1)x$. Similarlly we have $x(1\otimes b)=(1\otimes b)x$, therefore $x(a\otimes b)=x(a\otimes 1)(1\otimes b)=(a\otimes 1)(1\otimes b)x=(a\otimes b)x$. The last sentence yields that $x\in(A'\otimes_{alg}B')'=(A'\barotimes B')'=(A\barotimes B)''=A\barotimes B$, where the second equality is   \cite[Theorem IV.5.9]{Takesaki1}.
\end{proof}

\subsection{Normalizers, strong solidity and quasi-strong solidity}
Let $M$ be a finite von Neumann algebra.  Let $A \subseteq M$ 
be a von Neumann subalgebra of $M$. We set
\[
\begin{split}
\Nor_M(A) = & \{ u \in M \mid u \textrm{ unitary, } u A u^\ast = A \},   \\
\qNor^1_M(A) = & \{ x \in M \mid \exists x_1, \ldots, x_n \in M \textrm{ such that }  A x = \sum_{i=1}^n x_i A \}, \\
\qNor_M(A) = & \qNor^1_M(A) \cap  \qNor^1_M(A)^\ast,
\end{split}
\]
which are called the normalizer, one-sided quasi-normalizer and quasi-normalizer respectively. Note that $\Nor_M(A)$ is a group, $\qNor^1_M(A)$ is an algebra, and $\qNor_M(A)$ is a $\ast$-algebra. Further, we have inclusions $\Nor_M(A)\subseteq \qNor_M(A)\subseteq \qNor^1_M(A)$.

\begin{definition}[Strong solidity and quasi-strong solidity] A finite von Neumann algebra $M$ is called \textbf{strongly solid} if for any diffuse amenable  von Neumann subalgebra $A\subseteq M$, $\Nor_M(A)''$ is amenable.  $M$ is called \textbf{quasi-strongly solid} if for any diffuse amenable  von Neumann subalgebra $A\subseteq M$, $\qNor_M(A)''$ is amenable. 
    
\end{definition}
We remark that every quasi-strongly solid von Neumann algebra is strongly solid.

\subsection{Popa's intertwining-by-bimodule theory}
We recall the following definition due to Popa \cite{Popa2006a}, \cite{Popa2006b}. In this section we assume $M$ is a finite von Neumann algebra.

\begin{definition}[Embedding $A\prec_M B$] For two von Neumann subalgebras $A,B \subseteq M$, we will say that a corner of $A$ embeds in $B$ inside $M$ (denoted by $A\prec_M B$), if there exist projections $p\in A, q\in B$, a normal $\ast$-homomorphism $\theta:pAp\to qBq$ and a non-zero partial isometry $v\in qMp$ such that  $\theta(x)v=vx$ for all $x\in pAp$.
    
\end{definition}
\begin{definition}[Stable embedding $A\prec^s_M B$]
    For two von Neumann subalgebras $A,B \subseteq M$, we will say that $A$ embeds stably in $B$ inside $M$ (denoted by $A\prec_M^s B$) if for any projection $r\in A'\cap M$, we have $Ar\prec_M B$.
\end{definition}

\begin{lemma} [Lemma 2.4 in \cite{drimbePrimeII1Factors2019}, see also \cite{vaesExplicitComputationsAll2008}]\label{Lem=StableEmbedding}
    Let $(M, \tau)$ be a tracial von Neumann algebra and let $P,Q,R\subseteq M$ be von Neumann subalgebras. Then the following hold:
    \begin{enumerate}
        \item \label{Item=StableEmbedding:transative}Assume that $P\prec_M Q$ and $Q\prec_M^s R$. Then $P\prec_M R$;
        \item \label{Item=StableEmbedding:condition} Assume that, for any non-zero projection $z\in \Nor_{M}(P)'\cap M\subseteq \calZ(P'\cap M)$, we have $Pz\prec_M Q$. Then $P\prec_M^s Q$.
   \end{enumerate}
    In particular, we note that if $Q'\cap M$ is a factor and $P\prec_{M}Q$ and $Q\prec_{M} R$ then $P\prec_{M} R$.
\end{lemma}

\subsection{Simple graphs}
Let $\Gamma$ be a \textit{simple graph}, i.e. an undirected graph without double edges and without self-loops. We denote the vertex set of $\Gamma$ again by $\Gamma$. We write $v\in \Gamma$ for saying that $v$ is a vertex of $\Gamma$, and write $\Lambda\subseteq \Gamma$ for saying that $\Lambda$ is a subgraph of $\Gamma$ in case the vertex set of $\Lambda$ is a subset of the vertex set of $\Gamma$ and two vertices in $\Lambda$ share an edge if and only if they share an edge in $\Gamma$. For $v\in \Gamma$, we set two subgraphs
\begin{align}
    &\Link(v)=\{w\in \Gamma\mid \text{$v$ and $w$ share an edge}\},
    \\&\Star(v)=\{ v\}\cup \Link(v).
\end{align}
For $\Lambda\subseteq \Gamma$, we set $\Link(\Lambda)=\bigcap_{v\in \Lambda}\Link(v)$ and by convention we set $\Link(\emptyset)=\Gamma$. We denote $|\Gamma|$ the size of the graph, i.e. the number of vertices. We call a graph $\Gamma$ \textbf{connected} if it is non-empty and there exists a path between any two different vertices $v,w\in \Gamma$. A connected component of a graph $\Gamma$ is a subgraph $\Lambda\subseteq \Gamma$ that is connected and satisfies for any $v\in \Lambda$, $\Link(v)\subseteq \Lambda$. We call a graph $\Gamma$ \textbf{complete} if any two vertices in $\Gamma$ share an edge.
\begin{definition}
    We call a simple graph $\Gamma$ \textbf{locally finite}, if for every $v\in \Gamma$, $\Link(v)$ is finite.
\end{definition}
\begin{remark}
    Locally finite simple graphs have countably many vertices. 
\end{remark}
 \subsection{Particular graphs} 
 We denote $l_n, 2 \leq n < \infty$ for the finite line, i.e. the graph consisting of $n$ vertices labeled by $1, 2, \ldots, n$, and  $i,j \in l_n$ share an edge if and only if $\vert i - j \vert = 1$; and $l_{\infty}$ for the infinite line, i.e. the graph with vertex set $\mathbb{Z}$, and $i,j \in l_{\infty}$ share an edge if and only if $\vert i - j \vert = 1$.
 
 We denote $\mathbb{Z}_n,3\leq n<\infty$ for the cyclic graphs, i.e. the graph $l_n$ with an extra edge attached between $1$ and $n$.


\begin{figure}[h!]	
		\begin{tikzpicture}[baseline]
			\node at (-5,0) (1)	  	 {$1$};
			\node at (-4,0) (2)	  	 {$2$};
			\node at (-3,0) 	(3)  	 {$3$};
			\node at (-2,0) 	(4)  	 {$4$};
                \node at (-1,0) (5)	  	{$5$};

            			\node at (0.8,0.8) (6)	  	 {$5$};
			\node at (2,1.5) (7)	  	 {$1$};
			\node at (3.2,0.8) 	(8)  	 {$2$};
			\node at (1.3,-0.5) 	(9)  	 {$4$};
                \node at (2.7,-0.5) (10)	  	{$3$};

			\draw[-] (1) -- node {}(2);
                \draw[-] (2) -- node {} (3);
                \draw[-] (3) -- node {} (4);
                \draw[-] (4) -- node {} (5);

            \draw[-] (6) -- node {}(7);
                \draw[-] (7) -- node {} (8);
                \draw[-] (8) -- node {} (10);
                \draw[-] (10) -- node {} (9);
                \draw[-](9)-- node{}(6);

		\end{tikzpicture}

		\caption{The left is $l_5$, the right is $\mathbb{Z}_5$}
     \label{figure:graph-UPF-result}
\end{figure}
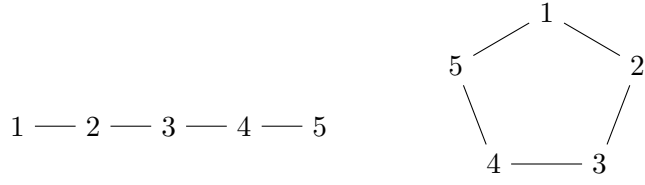

\subsection{Radius of graphs}
 
For a non-empty connected graph $\Gamma$, the  \textbf{radius of $\Gamma$} is defined as 
\[
    \Radius(\Gamma):=\inf_{t\in \Gamma}\sup_{s\in \Gamma}\Dist_{\Gamma}(t,s).
\] 
Here $\Dist_{\Gamma}(t,s)$ is the minimal length of a path in $\Gamma$ from $t$ to $s$. We set $\Radius(\emptyset) = 0$  and set $\Radius(\Gamma) = \infty$ when $\Gamma$ is not connected.

\subsection{Graph products}
 The graph product was first introduced as a basic construction for groups in Green's thesis \cite{greenGraphProductsGroups1990}. 
Given a simple graph $\Gamma$, and for every $v\in\Gamma$ given a group $G_v$, the graph product group $G_{\Gamma}=*_{v,\Gamma}G_v$ of $\{G_v\}_{v\in\Gamma}$  is defined as 
\[G_{\Gamma}:=*_{v\in\Gamma}G_v/ \langle sts^{-1}t^{-1}\mid s\in G_v,t\in G_w \text{ such that $v,w\in \Gamma$ 
 share an edge}\rangle.\]
 Here, $*_{v\in\Gamma}G_v$ is the free product of $\{G_v\}_{v\in\Gamma}$.

 When the graph $\Gamma$ is complete, $G_{\Gamma}=\prod_{v\in \Gamma}G_v$, that is, the graph product is the Cartesian product of groups. When the graph $\Gamma$ has no edges, $G_{\Gamma}= \ast_{v\in \Gamma}G_v$, that is, the graph product is the free product of groups. 
 
 Graph products of von Neumann algebras were first introduced by M\l{}otkowski \cite{mlotkowskiLfreeProbability2004}, as well as the  first author and Fima  \cite{caspersGraphProductsOperator2017}. For every $v\in\Gamma$, given a von Neumann algebra $M_v$ equipped with a faithful normal state $\varphi_v$, we can construct the canonical graph product $(M_{\Gamma}, \varphi_\Gamma)=*_{v,\Gamma}(M_v,\varphi_v)$. The construction of graph product of von Neumann algebras coincides with that of the graph product of groups, i.e. we have $\mathcal{L}(G_{\Gamma})=*_{v,\Gamma}\mathcal{L}(G_v)$. 
 The graph product of hyperfinite II$_1$-factor $R_{\Gamma}$ can therefore be defined as the group von Neumann  algebra $\mathcal{L}(G_{\Gamma})$ of the graph product over a graph $\Gamma$ of any countable amenable icc groups $G_v, v \in \Gamma$. We can set all these amenable icc group as $S_{\infty}$, the group of finite permutations of the natural numbers,   and we have $R_{\Gamma}=\mathcal{L}(*_{v,\Gamma}S_{\infty})$. We note that this definition  uses Connes' celebrated result that injectivity implies hyperfiniteness for separable II$_1$-factors \cite{ConnesClassificationInjectiveFactors}. 
 
 Similar to the graph product of groups, when the graph $\Gamma$ is complete, $M_{\Gamma}=\overline{\bigotimes}_{v\in\Gamma}M_v$; and when the graph $\Gamma$ has no edge, $(M_{\Gamma}, \varphi_\Gamma) =*_{v\in\Gamma}(M_v,\varphi_v)$ is the free product.  There are embeddings from every vertex von Neumann algebra $M_v$ to the graph product von Neumann algebra $M_{\Gamma}$ $\lambda_v:M_v\to M_{\Gamma}$.  
For $a\in \lambda_v(M_v)$ and $b\in \lambda_w(M_w)$, when $(v,w)\in E\Gamma$,  $a$ and $b$ commute; when $(v,w)\notin E\Gamma$, $a$ and $b$ are freely independent with respect to $\varphi_\Gamma$.  We refer to \cite{caspersGraphProductsOperator2017} for more knowledge about the concrete construction of graph product von Neumann algebras. We shall use several times that graph products of II$_1$-factors are again II$_1$-factors \cite{caspersGraphProductsOperator2017}. In particular  $R_\Gamma$ as defined above is a II$_1$-factor.  Further, if $M_v$ is a II$_1$-factor then for the relative commutant in $M_\Gamma$ we have
\[
M_v' \cap M_\Gamma = M_{\Link(v)}. 
\]

 The following proposition is about embeddings of von Neumann subalgebras inside graph product von Neumann algebras. This result was first proved in \cite[Proposition 5.9]{BCC}. 
 
\begin{proposition}\label{Lem=Selfembedding}
Let $\Gamma$ be a simple graph, and for $v\in \Gamma$ let $(M_{v},\tau_v)$  be a tracial von Neumann algebra. Fix $v\in \Gamma$ and let $N\subseteq M_v$ be diffuse. If $N \prec_{M_{\Gamma}} M_\Lambda$ for some subgraph $\Lambda \subseteq \Gamma$, then $v \in \Lambda$. In particular if $\Lambda = \{w\}$, a singleton set, then $v = w$.

\end{proposition}

\section{The Peterson-Thom conjecture and quasi-strong solidity of $R_{\Gamma}$ }
The Peterson-Thom conjecture, motived by Peterson and Thom their results on $L^2$-Betti numbers and further result from Popa's deformation/rigidity theory,  was first conjectured in \cite{PetersonThom}. This conjecture states that if $Q$ is a diffuse amenable von Neumann subalgebra of $\mathcal{L}(\mathbb{F}_n)$ $(n>1)$, then there is a unique maximal amenable von Neumann subalgebra $P\subseteq\mathcal{L}(\mathbb{F}_n)$ such that $Q\subseteq P$. Then, Hayes \cite{hayesRandomMatrixApproach2022} showed that the Peterson-Thom conjecture can be implied by the validity of a random matrix conjecture. Recently, this random matrix conjecture and therefore, the Peterson-Thom conjecture was solved independently by Belinschi-Capitaine \cite{BelinschiCapitaine22} and Bordenave-Collins\cite{BordenaveCollins23}. As an important and very natural application of the validity of the Peterson-Thom conjecture, in \cite{HayesJekelElayavalli}, Hayes, Jekel and Kunnawalkam Elayavalli showed that non-trival (interpolated) free group factors are quasi-strongly solid.  We state the next theorem for general interpolated free group factors $\mathcal{L}(\mathbb{F}_t), t\in(1,\infty)$ of which we omit the definition.  We shall only apply it in case of free group factors, i.e. when $t = n \in \mathbb{N}_{\geq 2}$. 

\begin{theorem}[Theorem 1.3 of \cite{HayesJekelElayavalli}]\label{thm3.1}
    Let $t\in(1,\infty)$ and $A\subseteq \mathcal{L}(\mathbb{F}_t)$ be a diffuse, amenable von Neumann subalgebra. Then for any subset $X\subseteq \qNor^1_M(A)$, $X''$ is amenable. In particular, $\mathcal{L}(\mathbb{F}_t)$ is quasi-strongly solid.
\end{theorem}
\begin{lemma}\label{Lem=FreeR}
    For $n\geq 2$ we have $R^{*n}\simeq\mathcal{L}(\mathbb{F}_n)$.
\end{lemma}
\begin{proof}
    From Theorem 4.1 of \cite{DykemaJFA1992}, we have $\mathcal{L}(\mathbb{F}_{n-1})*R\simeq\mathcal{L}(\mathbb{F}_{n}), n \geq 2$. And from Corollary 3.6 of \cite{DykemaPJM1994}, we have $R*R\simeq R*\mathcal{L}(\mathbb{Z})$. Combining these two formulas, we obtain the formula in this lemma.
\end{proof}
The following theorem will characterize when $R_{\Gamma}$ is quasi-strongly solid.
\begin{theorem}\label{thm3.3}
    Let $\Gamma$ be a finite simple graph. Then $R_{\Gamma}$ is quasi-strongly solid if only if every connected component of $\Gamma$ is complete.
\end{theorem}
\begin{proof}
    If  every connected components of $\Gamma$ is complete, then $R_{\Gamma}=R^{*n}$, where $n$ is the number of connected components of $\Gamma$; this follows   the fact that $R \overline{\otimes} R \simeq R$. When $n=1$, $R_{\Gamma}=R$, which itself is amenable,  therefore also quasi-strongly solid. When $n\geq 2$, by \cref{Lem=FreeR}, $R_{\Gamma}\simeq \mathcal{L}(\mathbb{F}_n)$, which is also quasi-strongly solid by \cref{thm3.1}.

    If $\Gamma$ has a incomplete component, then it must contain a $l_3$ as subgraph. But $R_{l_3}=R\barotimes \mathcal{L} \mathbb{F}_2$ is not strongly solid since $\mathcal{L}( \mathbb{F}_2)\subseteq \Nor_{R_{l_3}}(R)''$ where $\mathcal{L}( \mathbb{F}_2)$ is not amenable.
\end{proof}

\section{Main theorem}

This section contains our new results. We find a class of graph products of hyperfinite II$_1$-factors that remembers the graph.

\subsection{H-rigid graphs}
\begin{definition}
Let $\Gamma$ be a connected simple graph. A subset $\Gamma_0 \subseteq \Gamma$ is called an {\bf internal set} if $\Gamma_0\not =\emptyset$ and  $\Link(\Gamma_0)$ is not a complete graph or equivalently $\Link(\Gamma_0)$ contains at least 2 points that do not share an edge. We view $\Gamma_0$ as a subgraph of $\Gamma$ by declaring that two vertices in $\Gamma_0$ share an edge if and only if they share an edge in $\Gamma$. When an internal set is just a single vertex, we call it an {\bf internal vertex}.  Let $\Int(\Gamma)$ be the set of internal sets of $\Gamma$. When all internal sets of $\Gamma$ are internal vertices, we use $\Int(\Gamma)$ to denote the subgraph of $\Gamma$ whose vertices are internal vertices of $\Gamma$. In this case, we also call $\Int(\Gamma)$ the \textbf{internal graph} of $\Gamma$, and call the vertices in $\Gamma\setminus \Int(\Gamma)$  the \textbf{external vertices} of $\Gamma$.
\end{definition}

\begin{example}
Consider Figure \ref{figure:graph-UPF-result}. In $l_5$ the internal vertices are $2,3$ and $4$. In $\mathbb{Z}_5$ every vertex is internal. 
\end{example}

\begin{remark}
A non-empty set $\Gamma_0 \subseteq \Gamma$ is internal if and only if $R_{\Link(\Gamma_0)}$ is non-amenable. 
\end{remark}

\begin{proposition}\label{prop4.2}  
    Let $\Gamma$ be a non-empty connected graph with $\Int(\Gamma)\neq \emptyset$ whose internal sets are all internal vertices. Then every external vertex of $\Gamma$ must shares an edge with some internal vertex of $\Gamma$. 
\end{proposition}
\begin{proof}
    Let $v_0 \in \Gamma\setminus \Int(\Gamma)$ be an external vertex. As $\Gamma$ is connected we may let $V = (v_0, v_1, v_2, \ldots, v_n)$ be a shortest path from $v_0$ to $\Int(\Gamma)$; more precisely $V$ is a sequence with $n$ minimal such that $v_i$ and $v_{i+1}$ share an edge and $v_n \in \Int(\Gamma)$. By minimality of $n$ we have that $v_i$ and $v_j$ cannot share an edge if $\vert i - j \vert \geq 2$. Then, if $n \geq 2$ we have that $v_1$ is an internal vertex as $v_0$ and $v_2$ are neighbours that do not share an edge. This would mean that $v_1$ is an internal vertex contradicting minimality of $n$. So $n =1$.   
 

    
 \end{proof}

\begin{lemma}\label{lemma4.3}
    Let $\Gamma$ be a  connected graph whose  internal sets are all internal vertices. Then we have 
    \[\Radius(\Gamma)\leq\Radius( \Int(\Gamma))+1.\]
\end{lemma}
\begin{proof}
If $\Gamma=\emptyset$, the inequality holds; If $\Gamma\neq\emptyset$ and $\Int(\Gamma)=\emptyset$, then $\Gamma$ is complete, and $\Radius(\Gamma)=1$. Hence,  the inequality holds again.

Now suppose that $\Gamma\neq \emptyset$ and $\Int(\Gamma)\neq \emptyset$.
    For $t,s\in \Gamma$, by \cref{prop4.2} we can find an internal vertex $v$ such that $\Dist_{\Gamma}(v,s)\leq 1$. Then we have \[\Dist_{\Gamma}(t,s)\leq \Dist_{\Gamma}(t,v)+\Dist_{\Gamma}(v,s)
    \leq \Dist_{\Gamma}(t,v)+1.\]
    But since $\Dist_{\Gamma}(t,v)\leq \sup_{w\in\Int(\Gamma)}\Dist(t,w)$, we have
    \[
    \sup_{s\in\Gamma}\Dist_{\Gamma}(t,s)\leq\sup_{w\in\Int(\Gamma)}\Dist_{\Gamma}(t,w)+1.\]
    Finally, we obtain that 
    \begin{align*}
\Radius(\Gamma)&=\inf_{t\in\Gamma}\sup_{s\in\Gamma}\Dist_{\Gamma}(t,s)\\&\leq\inf_{t\in\Int(\Gamma)}\sup_{w\in\Int(\Gamma)}\Dist(t,w)+1\\&=\Radius(\Int(\Gamma))+1.
    \end{align*}
    This concludes the proof.
\end{proof}

\begin{definition}[H-rigid graphs]
    We call a simple graph $\Gamma$  {\bf H-rigid}  if  (1) it is locally finite; (2) all of its internal sets are internal vertices; (3) for every non-empty finite subgraph $\Gamma_0\subseteq \Gamma$ with $\Link(\Gamma_0)\neq \emptyset$, we have that every connected component  of $\Gamma_0$ is  complete. 
\end{definition}

\begin{proposition}\label{prop4.5.}
    
The following three kinds of graphs are H-rigid:

    \begin{enumerate}
        \item Lines   $l_n$ with $2\leq n \leq \infty$;
        \item Cyclic graphs $\mathbb{Z}_n$ with $3\leq n <\infty$;
        \item Locally finite trees.
    \end{enumerate}
    \end{proposition}
   \begin{proof}
        For lines $l_n$: When $n=2$, any non-empty subgraph of $l_2$ with non-empty link must be a singleton; when $3\leq n\leq\infty$, a subgraph with non-empty link must be a singleton or two vertices with one vertex in between, i.e. $(k,m) \in l_n \times l_n$ with   $\vert k - m \vert =2$. Such subgraphs clearly have the property that connected components are complete. Note that all vertex sets of two vertices with one vertex in between are not internal in $l_n, (3\leq n\leq \infty)$. Thus the internal sets of $l_n, (2\leq n\leq \infty)$ must be internal vertices.

        For cyclic graphs $\mathbb{Z}_n$: When $n=3$, any subgraph with non-empty link must be a $l_2$  or a singleton; when $4\leq n< \infty$, any subgraph with non-empty link must be a singleton or two vertices with one vertex in between. Again all such subgraphs have complete graphs as its connected components.  It is clear that $l_2$ is not an internal set in $\mathbb{Z}_3$ and  the vertex set of two vertices with one vertex between is not internal set in $\mathbb{Z}_n, (4\leq n < \infty)$. Therefore the internal sets in $\mathbb{Z}_n, (3\leq n < \infty)$ must be internal vertices.

        For locally finite trees: The subgraph with non-empty link must be several isolated vertices taken from a finite vertex set consisting of the parent of some node and all the offspring of that node. Again all such subgraphs have complete graphs as its connected components. Further, the vertex set of such subgraph will not be internal unless it is just a singleton since otherwise the link of such vertex set have only one vertex. Therefore the internal sets of a locally finite tree must be internal vertices.
    \end{proof}

   \begin{remark}
    H-rigid graphs may not be rigid (for the definition of rigid graphs see Definition 3.1 in \cite{BCC}). For example, $l_n, 3\leq n < \infty$ and finite trees are H-rigid but not rigid. In particular we obtain rigidity results for new types of graphs in this paper. 
\end{remark}

\subsection{Graph products of hyperfinite II$_1$-factors  over H-rigid graphs}

We set up the notation. Let $\Gamma$ and $\Lambda$ be simple graphs. For $v \in \Gamma, w \in \Lambda$ let $M_v = R, N_w = R$ where $R$ is the hyperfinite II$_1$-factor. For $\Gamma_0 \subseteq \Gamma, \Lambda_0  \subseteq \Lambda$, we set $M_{\Gamma_0} = \ast_{v, \Gamma_0} M_v$ and $N_{\Lambda_0} = \ast_{v, \Lambda_0} N_v$.  We sometimes simply write $M = M_\Gamma$ and $N=N_\Lambda$.

\begin{theorem}
\label{Prop=IntEmbedding}
Suppose that $\Gamma$ and $\Lambda$ are two simple graphs. Assume further that $\Lambda$ is  locally finite and  such that for every non-empty finite subgraph $\Lambda_0\subseteq \Lambda$ with $\Link(\Lambda_0)\neq \emptyset$, we have that every connected component of $\Lambda_0$ is complete. 
Suppose that $M := M_\Gamma \simeq N_\Lambda$ (notation before this theorem).  Then for every internal vertex $v \in \Int(\Gamma)$ there exists  an internal set $\Lambda_v \in \Int(\Lambda)$  such that $M_v \prec_M N_{\Lambda_v}$. 
\end{theorem}
\begin{proof}
    Since $v \in \Int(\Gamma)$ is an internal vertex, we have $M_v' \cap M= M_{\Link(v)}$ is not amenable. By Theorem I of \cite{BCC} there exists a $\Lambda_v \subseteq \Lambda$ with Link($\Lambda_v)\neq \emptyset$ such that  $M_{v} \prec_M N_{\Lambda_v}$.   As $\Lambda$ is locally finite and $\Link(\Lambda_v) \not =  \emptyset$, it must be the case that $\Lambda_v$ is finite.   We may assume that $\Lambda_v$ is minimal with this property meaning that there is no $\Lambda_v' \subsetneq \Lambda_v$ such that $M_v \prec_M N_{\Lambda_v'}$ 
    and note this $\Lambda_v$ is still a finite graph.

    We shall prove that $\Lambda_v \in \Int(\Lambda)$.  As $M_v \prec_M N_{\Lambda_v}$ there exist projections $p \in M_v, q \in N_{\Lambda_v}$, a non-zero partial isometry $w \in q M p$ and a normal $\ast$-homomorphism $\theta: p M_v p \rightarrow q N_{\Lambda_v} q$ such that $\theta(x) w = w x$ for all $x \in p M_{v} p$ and such that $\theta(p M_v p) \not\prec_M N_{\Lambda_v'}$ for any strict subgraph $\Lambda_v'$ of $\Lambda_v$ (see e.g. \cite[Lemma 1.4.5]{ioanaAmalgamatedFreeProducts2008}, \cite[Lemma 2.1]{BC}).

Take $u \in \Nor_M( M_v)$. Since $M_v$ is a factor, there exists a series of partial isometries $\{v_j\}_{j=1}^n$ in $M_v$ such that  $\sum_{j=1}^n v_j v_j^\ast = 1$ and $v_j^\ast v_j \leq p$ for every $j$. Then, 
    \[
        p u p (p M_v p) \subseteq p u M_v =  p M_v u \subseteq \sum_{j=1}^n (p M_v v_j) v_j^\ast u \subseteq  \sum_{j=1}^n (p M_v p) v_j^\ast u. 
    \]
    Similarly, or by taking adjoints, $(p M_v p)  p u p \subseteq  \sum_{j=1}^n   u v_j (p M_v p)$. It follows that $pup \in \qNor_{pMp}(p M_v p)$.
 
    Set $q_1 = \theta(p)$.   Note that $q_1 w = \theta(p) w = w p = w$ and so $q_1 \leq q$ and $w \in q_1 M p$.      
       For any $x \in \qNor_{pMp}(p M_v p)$ it follows by direct verification that $w x w^\ast \in \qNor_{q_1 M q_1}(\theta(p M_v p)) $. Indeed the assumption $x \in \qNor_{pMp}(p M_v p)$  implies that there are $\{x_i\}_{i=1}^n \subseteq p M p$ such that, 
    \[
     \theta(p M_v p) w x w^\ast = w p M_v p  xw^\ast \subseteq w \sum_{i=1}^n  x_i p M_v p w^\ast = \sum_{i=1}^n  w x_i w^\ast \theta(p M_v p), 
    \]
    and so  $w x w^\ast \in \qNor^1_{q_1 M q_1}(\theta(p M_v p)) $.   The inclusion  $w x w^\ast \in \qNor^1_{q_1 M q_1}(\theta(p M_v p))^\ast $ follows similarly.

    From the previous two paragraphs it thus follows that for $u \in \Nor_M( M_v)$ we have $w p u p w^\ast \in \qNor_{q_1 M q_1}(\theta(p M_v p))  $.
    
    Set $p_1 = w^\ast w \in (p M_v p)' \cap p M p$. So $p_1 \in \Nor_{pMp}(p M_v p)''$. Therefore we may define a $\ast$-homomorphism,
    \[
    \rho: p_1 \Nor_{M}( M_v )'' p_1 \rightarrow   \qNor_{q_1 M q_1}(\theta(p M_v p))'': x \mapsto w x w^\ast, 
    \]
    with $w \in  q_1 M p_1$. We note that $\rho$ is   injective.     
    Note further that 
    \[
     \Nor_{M}( M_v )''   =         M_{\Star(v)}   =   M_{v} \barotimes M_{\Link(v)}.
    \]
    
\noindent {\bf Claim:}  We have,
    \[
 \qNor_{q_1 M q_1}(\theta(p M_v p))'' =     \qNor_{q_1 N_{\Lambda_v} q_1}(\theta(p M_v p))''  \barotimes N_{\Link(\Lambda_v)}.    
    \]
    \begin{proof}[Proof of the claim]
    On the one hand, we have 
    $$\qNor_{q_1 N_{\Lambda_v} q_1}(\theta(p M_v p))'' \overline{\otimes} N_{\Link(\Lambda_v)}\subseteq \qNor_{q_1Mq_1}(\theta(pM_vp))''.$$
    On the other hand, by \cite[Proposition 5.8]{BCC},  we have 
    \begin{equation}\label{Eqn=EmbedQuasi}
    \qNor_{q_1Mq_1}(\theta(pM_vp))''\subseteq q_1 N_{\Lambda_v\cup \Link(\Lambda_v)}q_1=q_1N_{\Lambda_v}q_1\overline{\otimes} N_{\Link(\Lambda_v)}.
    \end{equation}
    Now take $x\in \qNor_{q_1Mq_1}(\theta(pM_vp))$. Then for any $y\in\theta(pM_vp)$, there are $x_1,\ldots, x_n\in q_1Mq_1$ and $y_1,\ldots, y_n\in \theta(pM_vp)$ such that 
    $yx = \sum_{i=1}^n x_iy_i$. Then for $\omega\in  (N_{\Link(\Lambda_v)})_*$, we have by viewing $x$ as an element of  $q_1N_{\Lambda_v}q_1\overline{\otimes} N_{\Link(\Lambda_v)}$ through \eqref{Eqn=EmbedQuasi}, that
    \begin{align*}
      y(id\otimes \omega)(x)=&(y\otimes1)(id\otimes \omega)(x) =(id\otimes \omega)((y\otimes 1)x) \\ 
      = & (id\otimes w)(\sum_{i=1}^n x_i(y_i\otimes 1)) =\sum_{i=1}^n(id\otimes \omega)(x_i)y_i. 
    \end{align*}
    Therefore, 
    \[(id\otimes \omega)(x)\in  \qNor_{q_1N_{\Lambda_v}q_1}(\theta(pM_vp))\subseteq  \qNor_{q_1N_{\Lambda_v}q_1}(\theta(pM_vp))''.\]     
Thus by  \cref{Lem=Slicing}, we have $x\in \qNor_{q_1 N_{\Lambda_v} q_1}(\theta(p M_v p))'' \overline{\otimes} N_{\Link(\Lambda_v)}$, which finishes the proof of the claim.
    \end{proof}
     
    In summary we have $p_1(M_v\barotimes M_{\Link(v)})p_1=M_v\barotimes p_1M_{\Link(v)}p_1$ and the injective *-homorphism
    \begin{equation}
    \rho: M_{v} \barotimes p_1 M_{\Link(v)}p_1 \rightarrow \qNor_{q_1 M_{\Lambda_v} q_1}(\theta(p M_v p))''  \barotimes N_{\Link(\Lambda_v)}.
    \end{equation}
         Since $M_{\Link(v)}$ is non-amenable, we have that $p_1 M_{\Link(v)}p_1$ is a non-amenable factor.   Therefore $M_{v} \barotimes p_1M_{\Link(v)}p_1$ is non-amenable and then  $\qNor_{q_1 M_{\Lambda_v} q_1}(\theta(p M_v p))''   \barotimes N_{\Link(\Lambda_v)}$ is also non-amenable. Now as $\theta(p M_v p)$ is amenable and $q_1 M_{\Lambda_v} q_1$ is  quasi-strongly solid by the assumption of $\Lambda$ and \cref{thm3.3}, it follows that $\qNor_{q_1 M_{\Lambda_v} q_1}(\theta(p M_v p))''  $ is amenable. Therefore  $N_{\Link(\Lambda_v)}$ must be non-amenable and hence $\Lambda_v \in \Int(\Lambda)$.   
\end{proof}

\begin{theorem}\label{thm4.7.} 
Let $\Lambda$ and $\Gamma$ be H-rigid graphs. 
Suppose that $M := M_\Gamma \simeq N_\Lambda$ (notation before Theorem \ref{Prop=IntEmbedding}). Then we have an isomorphism of graphs $\Int(\Gamma) \simeq \Int(\Lambda)$.  
\end{theorem}
\begin{proof}
First suppose that $\Int(\Gamma)=\emptyset$. Assume that $\Int(\Lambda)\neq \emptyset$, i.e. there exists $v_0\in \Int(\Lambda)$. By \cref{Prop=IntEmbedding}, there exists a non-empty subgraph $\Gamma_{v_0}\in \Int(\Gamma)$. Therefore $\Int(\Gamma)\neq \emptyset$ and we get a contradiction. Thus if $\Int(\Gamma)=\emptyset$, then $\Int(\Lambda)=\Int(\Gamma)=\emptyset$.

Now we assume that $\Int(\Gamma)\neq \emptyset$.
Take $v \in \Int(\Gamma)$ so that $M_v' \cap M = M_{\Link(v)}$ is non-amenable. By \cref{Prop=IntEmbedding} and the fact that all internal sets are singletons there exists an internal vertex $\alpha(v) \in \Int(\Lambda)$ such that $M_v \prec_M N_{\alpha(v)}$. Similarly for $w \in \Int(\Lambda)$ there exists an internal vertex $\beta(w) \in \Int(\Gamma)$ such that $N_w \prec_M M_{\beta(w)}$.  By   \cref{Lem=StableEmbedding}   it follows that $M_v \prec_M M_{\beta(\alpha(v))}$ and so $v = \beta(\alpha(v))$ by \cref{Lem=Selfembedding}. Similarly  $w =  \alpha(\beta(w))$. So $\alpha$ and $\beta$ are unique and inverses of each other. 

Take $v \in \Int(\Gamma)$.  We know that $N_{\alpha(v)}\prec_{M} M_{\beta(\alpha(v))} = M_v$. Taking relative commutants \cite[Lemma 3.5]{vaesExplicitComputationsAll2008}, using again factoriality of the vertex von Neumann algebras,  we find  
 $ M_{\Link(v)}   \prec_{M}   N_{\Link(\alpha(v))}$.
Now take $v' \in \Link(v) \cap \Int(\Gamma)$ so that the latter embedding gives  $M_{v'}  \prec_{M}   N_{\Link(\alpha(v))}$, hence $M_{v'}  \prec_{M}^s  N_{\Link(\alpha(v))}$ by \cref{Lem=StableEmbedding}. Then again by \cref{Lem=StableEmbedding}  we obtain
$N_{\alpha(v')}  \prec_{N}   N_{\Link(\alpha(v))}$.  This then implies by   \cref{Lem=Selfembedding} that
$\alpha(v')  \in    \Link(\alpha(v))$. So we conclude that $\alpha$ preserves edges. Similarly $\beta$ preserves edges, and it follows that $\alpha: \Int(\Gamma) \to \Int(\Lambda)$ is a graph isomorphism.

\end{proof}

  Using \cref{thm4.7.} we can classify $R_{\Gamma}$ for all H-rigid graphs satisfying $\Int(\Gamma) = \Gamma$, as well as $R_{l_n}$.

\begin{corollary}
Consider two H-rigid graphs $\Gamma$ and  $\Lambda$ such that $\Gamma = \Int(\Gamma)$ and $\Lambda = \Int(\Lambda)$. If $R_{\Gamma}\simeq R_{\Lambda}$, then $\Gamma\simeq \Lambda$. 
\end{corollary}

An infinite regular tree is a rooted tree in which the root has $d-1$ neighbors and any other point has $d$ neighbors.

\begin{example}
We give some typical examples of H-rigid graphs whose internal graphs are equal to themselves:

    \begin{enumerate}
        \item $\mathbb{Z}_n$, $n\geq 3$;
        \item All infinite trees without leaves whose root has more than one offspring;
        \item In particular, all infinite regular trees.
    \end{enumerate}
\end{example}

\begin{corollary}\label{coro4.10}
If $R_{l_n}\simeq R_{l_m}, 2 \leq n,m \leq \infty$, then m=n. 
\end{corollary}
\begin{proof}
    By \cref{prop4.5.}, we know that lines are H-rigid. Therefore by \cref{thm4.7.}, we have if $R_{l_n}\simeq R_{l_m}$, then $\Int(l_n)\simeq \Int(l_m)$.
    
    When $n=\infty$, if $R_{l_{\infty}}\simeq R_{l_m}$, then $\Int(l_m)\simeq\Int(l_{\infty})\simeq l_{\infty}$. We deduce that $m=\infty$. When $n<\infty$ we have $|l_n|=|\Int(l_n)|+2$. If $R_{l_n}\simeq R_{l_m}$, then $|\Int(l_n)|=|\Int(l_m)|<\infty$. Therefore, $m<\infty$ and $m=|l_m|=|\Int(l_m)|+2=|\Int(l_n)|+2=|l_n|=n$.   
\end{proof} 

In addition, using \cref{thm4.7.} we can also improve our previous graph radius rigidity result (Theorem F in \cite{BCC}) for graph products of hyperfinite II$_1$-factors over H-rigid graphs.
\begin{corollary}\label{coro4.11}
    For two H-rigid graphs $\Gamma,\Lambda$, if $R_{\Gamma}\simeq R_{\Lambda}$, then \[|\Radius(\Gamma)-\Radius(\Lambda)|\leq 1.\]
\end{corollary}
\begin{proof}
Without loss of generality, we may assume that $\Radius(\Gamma)\geq \Radius(\Lambda)$.
    By \cref{lemma4.3}, we have 
    \begin{align*}    
    |\Radius(\Gamma)-\Radius(\Lambda)|&=\Radius(\Gamma)-\Radius(\Lambda)\\&\leq\Radius(\Int(\Gamma))+1-\Radius(\Lambda)\\&=\Radius(\Int(\Lambda))+1-\Radius(\Lambda)\leq 1.
    \end{align*}
   This concludes the proof.  
\end{proof}

\printbibliography

\end{document}